\documentclass[10pt]{amsart}

\usepackage{amsthm}
\usepackage{amssymb}
\usepackage{amsmath}
\usepackage{graphicx}
\usepackage{amscd}
\usepackage{amsfonts}
\usepackage{amsbsy}
\usepackage[T1]{fontenc}
\usepackage[english]{babel}

\usepackage{epsfig}
\usepackage{amssymb}

\newtheorem{thm}{Theorem}
\newtheorem{cor}[thm]{Corollary}

\newtheorem{prop}[thm]{Proposition}

\newtheorem{lemma}[thm]{Lemma}

\newcommand{\cl}{\operatorname{Cl}}

\newcommand{\supp}{\operatorname{supp}}

\newcommand{\si}{\operatorname{Sing}}

\title[Homoclinic orbits and entropy for three-dimensional flows]{Homoclinic orbits and entropy for three-dimensional flows}

\author{A.M. Lopez, R.J. Metzger, C.A. Morales}

\address{Institute of Exact Sciences (ICE), Universidade Federal Rural do Rio de Janeiro, 23890-000 Seropedica, Brazil.}

\email{barragan@im.ufrj.br.}

\address{Instituto de Matem\'atica y Ciencias Afines (IMCA), Universidad Nacional de Ingenier\'{i}a,
Calle Los Bi\'ologos 245, Urb. San C\'esar La Molina
Lima 12, Lima, Peru.}

\email{metzger@imca.edu.pe.}

\address{Instituto de Matem\'atica, Universidade Federal do Rio de Janeiro, P. O.
Box 68530, 21945-970 Rio de Janeiro, Brazil.}

\email{morales@impa.br.}
\keywords{Hyperbolic ergodic measure, Lyapunov exponents, Flow.}

\thanks{Partially supported by MATHAMSUB 15 MATH05-ERGOPTIM, Ergodic Optimization of Lyapunov Exponents.}

\subjclass[2010]{Primary 37D25; Secondary 37C40}

\begin{document}

\begin{abstract}
We prove that every $C^1$ three-dimensional flow with positive topological entropy can be
$C^1$ approximated by flows with homoclinic orbits.
This extends a previous result for $C^1$ surface diffeomorphisms \cite{g}.
\end{abstract}

\maketitle

\section{Introduction}
\noindent
In his classical paper \cite{k} Katok proved
that every $C^{1+\alpha}$ surface diffeomorphism with positive topological entropy
has a homoclinic orbit.
In \cite{g} Gan asked if this result is true for $C^1$ surface diffeomorphisms too.
He didn't answer this question but managed to prove that
every $C^1$ surface diffeomorphism with positive entropy can be
$C^1$ approximated by diffeomorphisms with homoclinic orbits.
More recently, the authors \cite{gy} proved that
every three-dimensional flow can be $C^1$ approximated by Morse-Smale flows or by flows with
a homoclinic orbit (this entails the weak Palis conjecture for three-dimensional flows).
From this they deduced that
there is an open and dense subset of three-dimensional flows
where the property of having zero topological entropy is invariant under topological equivalence.
Moreover, the $C^1$ approximation by three-dimensional flows with robustly zero topological entropy is equivalent
to the $C^1$ approximation by Morse-Smale ones.

In this paper we will extend \cite{g} from surface diffeomorphisms to three-dimensional flows.
In other words, we will prove that every $C^1$ three-dimensional flow with positive topological entropy
can be $C^1$ approximated by flows with homoclinic orbits.
Let us state our result in a precise way.

The term {\em flow} will be referred to $C^1$ vector fields $X$ defined on a compact
connected boundaryless Riemannian manifold $M$. To emphasize its differentiability we say that $X$ is a $C^r$ flow, $r\in\mathbb{N}^+$.
When $dim(M)=3$ we say that $X$ is a {\em three-dimensional flow}.
The flow of $X$ will be denoted by $\phi_t$ (or $\phi^X_t$ to emphasize $X$), $t\in\mathbb{R}$.
We denote by $\Phi_t=\Phi_t$ the derivative of $\phi_t$.
The space of $C^r$ flows $\mathcal{X}^r$ is endowed with the standard $C^r$ topology.
We say that $x\in M$ is a {\em periodic point} of a flow $X$ if there is a minimal positive number $\pi(x)$ (called {\em period})
such that $\phi_{\pi(x)}(x)=x$.
Notice that $1$ is always an eigenvalue of the
derivative $DX_{\pi(x)}(x)$ with eigenvector $X(x)$.
The remainders eigenvalues will be referred to as the eigenvalues of $x$.
We say that the orbit $O(x)=\{X_t(x):t\in\mathbb{R}\}$ of a periodic point $x$ (or the periodic point $x$)
is {\em hyperbolic} if it has no eigenvalue of modulus $1$.
In case there are eigenvalues of modulus less and bigger than $1$
we say that the hyperbolic periodic point is a {\em saddle}. 

The Invariant Manifold Theory \cite{hps} asserts that through any periodic saddle $x$ it passes a pair of invariant manifolds,
the so-called strong stable and unstable manifolds
$W^{ss}(x)$ and $W^{uu}(x)$, tangent at $x$ to the eigenspaces corresponding to the eigenvalue
of modulus less and bigger than $1$ respectively.
Saturating them with the flow we obtain the stable and unstable manifolds $W^s(x)$ and $W^u(x)$ respectively.
We say that $O$ is a {\em homoclinic orbit} (associated to a periodic saddle $x$)
if $O\subset W^s(x)\cap W^u(x)\setminus O(x)$. If, additionally, $\dim(T_qW^s(x)\cap T_qW^u(x))\neq1$
then we say that $O$ is a {\em homoclinic tangency}.

We say that $E\subset X$ is {\em $(T,\epsilon)$-separated} for some $T,\epsilon>0$
if for any distinct point $x,y\in E$ there exists $0\leq t\leq T$ such that
$d(X_t(x),X_t(y))>\epsilon$.
The number
$$
h(X)=\lim_{\epsilon\to0}\limsup_{T\to\infty}\frac{1}T\log \sup\{Car(E):E \mbox{ is } (T,\epsilon)\mbox{-separated}\}
$$
is the so-called {\em topological entropy} of $X$.
With these definitions we can state our result.

\begin{thm}
\label{thAA}
Every $C^1$ three-dimensional flow with positive topological entropy can be $C^1$ approximated by
flows with homoclinic orbits.
\end{thm}

The proof follows Gan's arguments \cite{g} using
the variational principle (e.g. \cite{bru}) and Ruelle's inequality.
But we simplify such arguments by using recent tools as a flow-version of a result by Crovisier \cite{c} and Gan-Yang \cite{gy}.

Denote by $\cl(\cdot)$ and $int(\cdot)$ the closure and interior operations.
As in \cite{g} we get from Theorem \ref{thAA} the following corollary.

\begin{cor}
\label{c1}
If $\mathcal{H}_+=\{X\in \mathcal{X}^1:h(X)>0\}$, then $\cl(int(H_+))=H_+$
and so $\mathcal{H}_+$ has no isolated points.
\end{cor}

\section{Proof of Theorem \ref{thAA}}
\label{sec1}

\noindent
Denote by $\si(X)$ the set of singularities of a flow $X$.
Given $\Lambda\subset M$, we denote by $\Lambda^*=\Lambda\setminus\si(X)$ the set of regular points
oif a flow $X$ in $\Lambda$.
Define by $E^X$ the map assigning to $p\in M$ the subspace of $T_pM$ generated by $X(p)$.
It turns out to be a one-dimensional subbundle of $TM$ when restricted to $M^*$.
Define also the normal subbundle $N$ over $M^*$ whose fiber
$N_p$ at $p\in M^*$ is the orthogonal complement of $E^X_p$ in $T_pM$.
Denoting by $\pi=\pi_p:T_pM\to N_p$ the orthogonal projection we obtain the {\em linear Poincar\'e flow}
$\psi_t:N\to N$ defined by $\psi_t(p)=\pi_{\phi_t(p)}\circ \Phi_t(p)$.
When necessary we will use the notation
$N^X$ and $\psi^X_t$ to indicate the dependence on $X$.

For a (nonnecessarily compact) invariant set $\Omega\subset M^*$, one says that {\em $\Omega$
has a dominated splitting with respect to the Poincar\'e flow} if there are a continuous splitting
$N_\Omega=N^-\oplus N^+$ into $\psi_t$-invariant subbundles $N^-,N^+$ and positive numbers $K,\lambda$ such that
$$
\|\psi_t|_{N^-_x}\|\cdot\|\psi_{-t}|_{N^+_{\phi_t(x)}}\|\leq Ke^{-\lambda t},
\quad\quad\forall x\in \Omega, t\geq0.
$$

Let $\mu$ be a Borel probability measure of $M$.
We say that $\mu$ is {\em nonatomic} if it has no points with positive mass.
We say that $\mu$ is supported on $H\subset M$ if $\supp(\mu)\subset H$, where $\supp(\mu)$
denotes the support of $\mu$.
We say that $\mu$ is
{\em invariant} if
$\mu(X_t(A))=\mu(A)$ for every borelian $A$ and every $t\in\mathbb{R}$.
Moreover, $\mu$ is {\em ergodic} if it is invariant and
every measurable invariant set has measure $0$ or $1$.

Oseledets's Theorem \cite{s} ensures that
every ergodic measure $\mu$ is equipped with
an invariant set of full measure $R$, a positive integer $k$, real numbers
$\chi_1<\chi_2<\cdots<\chi_{k}$ and a measurable invariant splitting $T_RM=
E^1\oplus \cdots \oplus E^k$ over $R$
such that
$$
\displaystyle\lim_{t\to\pm\infty}\frac{1}{t}\log\|\Phi_t(x) e^i\|=\chi_i,
\quad\quad\forall x\in R, \forall e^i\in E^i_x\setminus\{0\}, \forall 1\leq i\leq k.
$$
The numbers $\chi_1,\cdots, \chi_k$ are the so-called {\em Lyapunov exponents} of $\mu$.
Clearly, the Lyapunov exponent of $\mu$ corresponding to the flow
direction is zero. If the remainders exponents are nonzero, then we say that $\mu$ is a {\em hyperbolic measure}.
In such a case we can rewrite the Oseledets decomposition of $\mu$
as $T_RM=E^s\oplus E^X\oplus E^u$ where $E^s$ (resp. $E^u$) is the sum of the subbundles
$\{E_1,\cdots, E^k\}$ for which the corresponding Lyapunov exponent is less than (resp. bigger than) $1$.
We then say that the {\em Oseledets decomposition of $\mu$ is dominated with respect to the Poincar\'e flow}
if $\supp(\mu)^*\neq\emptyset$ (equivalently $\mu(\si(X))=0$)
and the decomposition $N_R=N^s\oplus N^u$
given by $N^*=\pi(E^*)$ for $*=s,u$ is dominated with respect to the Poincar\'e flow.

We shall use the following lemma.

\begin{lemma}
\label{cro1}
Let $\mu$ be a hyperbolic ergodic measure of a flow $X$
whose Oseledets decomposition is dominated with respect to the Poincar\'e flow.
Then, there are $\eta,T>0$ such that $\mu$ is ergodic for $\phi^X_T$,
$$
\int\log\|\psi_T|_{N^s}\|d\mu\leq-\eta\quad \quad \mbox{ and }\quad \quad \int\log\|\psi_{-T}|_{N^u}\|d\mu\leq-\eta.
$$
\end{lemma}

\begin{proof}
It follows from the hypothesis that $\mu(\si(X))=0$.
On the other hand, $\mu$ is ergodic for $X$ so there is $T_1>0$ such that $\mu$
is totally ergodic for $\phi_{T_1}$ (c.f. \cite{ps}).
Since $\mu$ is hyperbolic, there is $\eta_0>0$ such that any
Lyapunov exponent off the flow direction belongs to $\mathbb{R}\setminus [-\eta_0,\eta_0]$.
From this and the
Furstenberg-Kesten Theorem (see also p. 150 in \cite{w})
we obtain
$$
\lim_{n\to\infty}\frac{1}n\log\|\psi_{nT_1}|_{N^s_x}\|\leq -\eta_0\mbox{ and }
\lim_{n\to\infty}\frac{1}n\log\|\psi_{-nT_1}|_{N^s_{\phi_{nT_1}(x)}}\|\leq -\eta_0,
$$
for $\mu$-a.e. $x\in M$.
Hence
$$
\lim_{n\to\infty}\frac{1}n \int\log\| \psi_{nT_1}|_{N^s}\|d\mu\leq-\frac{\eta_0}2\quad\mbox{ and }\quad
\lim_{n\to\infty}\frac{1}n \int\log\|\psi_{-nT_1}|_{N^u}\|d\mu\leq -\frac{\eta_0}2
$$
by the Dominated Convergence Theorem.
Now take $T=nT_1$ and $\eta=n\frac{\eta_0}2$ with $n$ large.
\end{proof}

Denote by $\cl(\cdot)$ the closure operation.
We say that $H\subset M$ is a {\em homoclinic class} of $X$ if there is a periodic saddle $x$
such that
$$
H=\cl(\{q\in W^s(x)\cap W^u(x): \dim(T_qW^s(x)\cap T_qW^u(x))=1\}).
$$
A homoclinic class is {\em nontrivial} if it does not reduce to a single periodic orbit.

The following is the flow-version of Proposition 1.4 in \cite{c}.

\begin{prop}
\label{thA}
For every flow,
every hyperbolic ergodic measure whose
Oseledets decomposition is dominated with respect to the Poincar\'e flow is supported on a homoclinic class.
\end{prop}

\begin{proof}
Let $\mu$ be a hyperbolic ergodic measure of a flow $X$.
Suppose that the Oseledets decomposition of $\mu$ is dominated with respect to the linear Poincar\'e flow.
By Lemma \ref{cro1} there are $\eta,T>0$ such that $\mu$ is ergodic for $\phi^X_T$,
$$
\int\log\|\psi_T|_{N^s}\|d\mu\leq-\eta\quad \quad \mbox{ and }\quad \quad \int\log\|\psi_{-T}|_{N^u}\|d\mu\leq-\eta.
$$
It follows from the hypothesis that $\mu(\si(X))=0$. Since $\mu$ is ergodic, we obtain
$$
\int\log\|\Phi_T|_{E^X}\|d\mu=0.
$$
Replacing in the two previous inequalities we obtain
$$
\int\log\|\psi^*_T|_{N^s}\|d\mu\leq-\eta\quad \quad \mbox{ and }\quad \quad \int\log\|\psi^*_{-T}|_{N^u}\|d\mu\leq-\eta,
$$
where
$$
\psi^*_t=\frac{\psi_t}{\|\Phi_t(x)|_{E^X_x}\|},
\quad\quad x\in M^*, t\in \mathbb{R}
$$
is the scaled linear Poincar\'e flow (c.f. \cite{sgw}).

On the other hand, standard arguments (c.f. \cite{lgw}) imply that the decomposition $N_R=N^s\oplus N^u$
(which is dominated for the Poincar\'e flow by hypothesis) extends continuously to a dominated splitting $N_{\supp(\mu)^*}=N^s\oplus N^u$
with respect to the linear Poincar\'e flow.
By Lemma 2.29 in \cite{ap}
there are a neighborhood $U$ of
$\supp(\mu)$ and a splitting $N_{\Lambda^*}=N^s\oplus N^u$
extending $N_{\supp(\mu)^*}=N^s\oplus N^u$
where $\Lambda=\bigcap_{t\in\mathbb{R}}X_t(U)$.

From this point forward we can reproduce the arguments on p. 214 of \cite{sgw} to conclude the proof.
\end{proof}

\begin{proof}[Proof of Theorem \ref{thAA}]
Let $X$ be a three-dimensional flow with positive topological entropy.
By the variational principle (e.g. \cite{bru})
there is an invariant measure $\mu$ of $X$ such that $h_\mu(X_1)>0$,
where $h_\mu$ is the metric entropy operation. By the ergodic decomposition theorem we can assume that $\mu$ is ergodic.

By Ruelle's inequality (e.g. Theorem 5.1 in \cite{g}) we get that $\mu$ has at least one positive Lyapunov exponent.
By applying this inequality to the reversed flow we obtain that $\mu$ has also a negative exponent.
Since $dim(M)=3$, we conclude that $\mu$ is hyperbolic of {\em saddle-type}
(i.e. with positive and negative exponents).

By the Ergodic Closing Lemma for flows (c.f. Theorem 5.5 in \cite{sgw})
there are a sequence of flows $X^n$ and a sequence of hyperbolic periodic orbits $\gamma_n$ of $X_n$
such that $X_n\to X$ and $\gamma_n\to \supp(\mu)$ as $n\to \infty$ where the latter convergence is with respect to the Hausdorff
topology of compact subsets of $M$.
By passing to a subsequence if necessary we can assume that the index (stable manifold dimension) of these periodic orbits is constant ($i$ say).

Now we assume by contradiction that $X$ cannot be approximated by flows with homoclinic orbits.
Hence $X$ cannot be approximated by flows with homoclinic tangencies either.

Since $dim(M)=3$, $i$ can take the values $0,1,2$ only.
If $i=2$ then each $\gamma_n$ is an attracting periodic orbit of $X^n$.
Since $X$ cannot be approximated by flows with homoclinic tangencies, Lemma 2.9 in \cite{gy}
implies that there is $T>0$ such that
$\|\psi^{X^n}_T|_{N^{X^n}_x}\|\leq \frac{1}2$ for all $n\in\mathbb{N}$ and all $x\in \gamma_n$.
Letting $n\to\infty$ we get $\|\psi^X_T|_{N_x}\|\leq\frac{1}2$ for all $x\in \supp(\mu)$.
This would imply that the Lyapunov exponents of $\mu$ off the flow direction are all negative.
Since $\mu$ is saddle-type, we obtain a contradiction proving $i\neq 2$. Similarly, $i\neq0$ and so $i=1$.
This allows us to apply Corollary 2.10 in \cite{gy} to obtain a dominated splitting $N_{\supp(\mu)^*}=N^-\oplus N^+$ of index $1$ (i.e.
$dim(N^-)=1$) with respect to the Poincar\'e flow.

Next we observe that both the Oseledets splitting $N^s\oplus N^u$ for the linear Poincar\'e flow
and the splitting $N^-\oplus N^+$ obtained above are pre-dominated of index $1$ in the sense
of Definition 2.1 in \cite{lgw}.
Since 
pre-dominated splittings of prescribed index are unique (c.f. Lemma 2.3 in \cite{lgw}),
we get $N^s\oplus N^u=N^-\oplus N^+$.

Since $N^-\oplus N^+$ is dominated with respect to the Poincar\'e flow,
the Oseledets decomposition $N^s\oplus N^u$ of $\mu$ is dominated with respect to the linear Poincar\'e flow either.
We conclude that $\mu$ is supported on a homoclinic class
by Proposition \ref{thA}.

Since $\mu$ has positive metric entropy, such a homoclinic class is nontrivial and so
$X$ has a homoclinic orbit against the assumption.
This contradiction completes the proof of the theorem.
\end{proof}

\end{document}